\documentclass[11pt]{amsart}
\usepackage{amsmath}
\usepackage{amsfonts}
\usepackage{amsthm}
\usepackage{amssymb}
\usepackage{verbatim}
\usepackage{enumerate}
\usepackage{graphicx}
\usepackage{longtable}
\usepackage[all,cmtip]{xy}
\usepackage{upref}
\usepackage{tikz}
\usepackage{hyperref}
\usepackage{color}
\usepackage[hang,small,bf]{caption}

\numberwithin{equation}{section}

\newtheorem{defn}{Definition}[section]
\newtheorem{teo}[defn]{Theorem}
\newtheorem{prop}[defn]{Proposition}
\newtheorem{oss}[defn]{Remark}
\newtheorem{lemma}[defn]{Lemma}

\newtheorem*{teo*}{Theorem}
\newtheorem*{prop*}{Proposition}

\renewcommand{\theta}{\vartheta}

\newcommand{\Z}{{\mathbb Z}}
\newcommand{\R}{{\mathbb R}}

\newcommand{\HH}{{\mathbb H}}

\newcommand{\dt}{\frac{d}{dt}}

\renewcommand{\phi}{\varphi}

\begin{document}

\title{Smooth Hamilton-Jacobi solutions for the horocycle flow}

%    author information
\author{Luca Asselle}
\address{Ruhr-Universit\"at Bochum, Fakult\"at f\"ur Mathematik, NA 4/35, Universit\"atsstra\ss e 150, D-44780 Bochum, Germany}
\email{\href{mailto:luca.asselle@ruhr-uni-bochum.de}{luca.asselle@ruhr-uni-bochum.de}}
%\thanks{}


\date{\today}
%\dedicatory{}
\begin{abstract}
In this paper we compute all the smooth solutions to the Hamilton-Jacobi equation associated with the horocycle flow. This can be seen 
as the Euler-Lagrange flow (restricted to the energy level set $E^{-1}(\frac 12)$) defined by the Tonelli Lagrangian $L:T\HH\rightarrow \R$ given by (hyperbolic) 
kinetic energy plus the standard magnetic potential. The method we use is to look at Lagrangian graphs that are contained in the level set $\{H=\frac 12\}$, where 
$H:T^*\HH\rightarrow \R$ denotes the Hamiltonian dual to $L$.
\end{abstract}

\maketitle

%%%%%%%%%%%%%%%%%%%%%%%%%%%%%%%%%%%%%%%%%%%%%%%%%%%%%%
%%%%%%%%%%%%%%%%%%%%%%%%%%%%%%%%%%%%%%%%%%%%%%%%%%%%%%
%%%%%%%%%%%%%%%%%%%%%%%%%%%%%%%%%%%%%%%%%%%%%%%%%%%%%%

\section{Introduction}
The horocycle flow was first introduced by Hedlund in his celebrated paper \cite{Hed32}; this flow, defined on the unit tangent bundle of the hyperbolic plane $\HH$ 
shifts inward pointing normal vectors to the horospheres (i.e. circles that are tangent to the boundary $\partial \HH$) to the left with unit speed. 
Its peculiarity relies on the fact that, once projected to a compact quotient of $\HH$ (that is, to a closed orientable surface with genus gretar than or equal to 2), 
it becomes \textit{minimal}, meaning that every orbit is dense (cf. \cite{Hed32} or \cite{BKS91}). In particular, this flow has no periodic orbits. 

The horocycle flow can be seen as (or, more precisely, it is conjugated via the Legendre transform to) the Hamiltonian flow at level $\frac 12$ defined by the Hamiltonian $H'(q,p)=\frac 12 \|p\|_q^2$ and by the twisted symplectic form $\omega_\sigma=dp\wedge dq + \pi^*\sigma$, where $\|\cdot\|_q$ is the dual norm on $T^*\HH$ induced by the hyperbolic metric of constant curvature $-1$, $\sigma$ is the standard area form on $\HH$ and $\pi:T^*M\rightarrow M$ is the cotangent bundle map. Equivalently, it can be seen as the Hamiltonian 
flow at level $\frac 12$ associated with the Hamiltonian 
$$H(q,p) = \frac 12 \, \|p-\eta\|_q^2$$
and with the standard symplectic form $dp\wedge dq$, where $\eta$ is a primitive of the standard area form $\sigma$. The projection of this flow to a 
compact quotient of $\HH$ represents, up to now, the only example of Tonelli Hamiltonian flow on a twisted cotangent bundle over a compact configuration space 
with an energy level without periodic orbits (the projection of the area form is not exact and hence in this case one has no representation of the horocycle flow as a Hamiltonian flow on the standard cotangent bundle).
Recall that a Tonelli Hamiltonian on a closed Riemannian manifold $M$ is a smooth function $H:T^*M\rightarrow \R$ which is superlinear and $C^2$-strictly convex in the fibers. 
If one allows the Hamiltonian to be more general than Tonelli, then many other counterexamples to the existence of periodic orbits are provided for instance in \cite{GG04}.

In this paper we study the existence of smooth solutions to the Hamilton-Jacobi equation 
\begin{equation}
H(q,d_qu)=\frac 12
\label{hamJac}
\end{equation}
associated with the horocycle flow. The method we will employ to compute the solutions of \eqref{hamJac} is to look at exact Lagrangian graphs (cf. Section \ref{preliminaries} for the definition) contained in $\{H=\frac 12\}$. We will perform this study in Section \ref{proofteogrosso} obtaining the following:

\begin{teo} The following statements hold: 
\\
\begin{enumerate}
\item If $u$ is a smooth solution of (\ref{hamJac}), then $u$ is either constant or (up to the addition of a constant) of the form 
$$u_a(x,y)=2\,\arctan \left (\frac{x-a}{y}\right ), \quad \text{for some}\ a\in \R.$$
\item If the graph $G_\alpha$ of a 1-form $\alpha$ is contained in the energy level $H^{-1}(\frac 12)$ and is invariant und the Hamiltonian flow, then $G_\alpha$ is exact Lagrangian (i.e. $\alpha$ is exact).
\end{enumerate}
\label{teogrosso}
\end{teo}

Theorem \ref{teogrosso}, \textit{(2)} represents the converse of the well-known Hamiltoni-Jacobi theorem for graphs (cf. Theorem \ref{teohj} for the general statement),
which states that any Lagrangian graph contained in a level set $\{H=k\}$ is invariant.

The same method will be then used in Section \ref{geodesics} to compute some particular solution of the Hamilton-Jacobi equation associated with the geodesic flow of the 
hyperbolic plane. 

\vspace{5mm}

\noindent \textbf{Acknowledgments}: This paper has arisen as part of my master thesis project, that was performed during my stay at the Rheinische Friedrich-Wilhelms Universit\"at Bonn. I would like to thank Prof. Abbondandolo for drawing my attention to the subject and the Universit\"at Bonn, which provided a very stimulating working environment.

%%%%%%%%%%%%%%%%%%%%%%%%%%%%

\section{Preliminaries}
\label{preliminaries}
On the hyperbolic plane $\HH=\{(x,y)\in \R^2|y>0\}$ we consider the standard Riemannian metric of constant curvature $-1$ given by
$$g(x,y)=\frac{1}{y^2}(dx^2 +dy^2).$$

We denote with $(q,v)$ the elements in $T\HH$ and with $v_x,v_y$ the components of a tangent 
vector $v\in T_q\HH$ in the $x$-direction, $y$-direction respectively. Consider the Tonelli Lagrangian $L:T\HH\rightarrow \R$ given by
\begin{equation}
L(q,v)\ = \frac{1}{2}\|v\|_q^2 +\eta_q(v),
\label{Lag}
\end{equation}
where $\|v\|_q :=\sqrt{g_q(v,v)}$ is the norm induced by $g$ and $\eta$ is the 1-form on $\HH$
$$\eta_{q}=\eta_{(x,y)}:=\ \frac{dx}{y}, \quad \forall q=(x,y)\in T\HH.$$
Observe that the differential of $\eta$ is the area form on $\HH$ induced by the Riemannian metric $g$. The Euler-Lagrange flow associated with $L$ models the motion of a unit mass and charge particle on $\HH$ under the effect of the magnetic potential $\eta$. In the literature, such flows are called \textit{magnetic flows} (see for instance \cite{Arn61}, \cite{Mer10}, \cite{Abb13}, \cite{AB15a},\cite{AB15b}). Being $L$ time-independent, the energy function associated with it
\begin{equation}
E(z,v)\ =\ \frac{1}{2} \, \|v\|_q^2
\label{energia}
\end{equation}
is an integral of the motion, i.e. invariant along the solutions of the Euler-Lagrange equation 
$\nabla_t\dot{\gamma}=Y_\gamma(\dot{\gamma}).$ Here $\nabla_t$ denotes the covariant derivative with respect to $g$ and $Y:T\HH\rightarrow T\HH$ is the bundle map (the \textit{Lorentz force}) defined by 
\begin{equation}
d\eta_q(u,v)=g_q\big (Y_q(u),v\big ), \quad \forall u,v\in T_q\HH, \ \forall q\in\HH.
\label{bundlemap}
\end{equation}

If we denote by $(u_x,u_y),\, (v_x,v_y),\, (w_x,w_y)$ the components of $u, v, Y_q(u)$ respectively, then a simple computation shows that (\ref{bundlemap}) is equivalent to
$$u_xv_y - u_yv_x=w_xv_x+w_yv_y,\quad \forall u_x,u_y,v_x,v_y.$$ 
In particular, the equality above for $v_x=1$ and $v_y=0$ (respectively for $v_x=0$ and $v_y=1$) yields $w_x=-u_y$ (respectively $w_y=u_x$), so that $Y_q(u)=J u$, where $J$ denotes the $2\times 2$-simplectic identity. Hence
$$\nabla_t \, \dot{\gamma}\ =\ J \, \dot{\gamma}$$
is the Euler-Lagrange equation associated with the Lagrangian $L$ in (\ref{Lag}). Since $\HH$ is an open subset of $\R^2$ the Euler-Lagrange equation can also be written using standard coordinates. Namely, the Lagrangian $L$ in coordinates is given by
\begin{equation}
L\big ((x,y),(v_x,v_y)\big )= \frac{1}{2y^2} \big (v_x^2+v_y^2\big ) + \frac{v_x}{y}
\label{Lagrangian}
\end{equation}
and this yields
\begin{equation}
\left \{ \begin{array}{l} \displaystyle \dt \left (\frac{v_x}{y^2} + \frac{1}{y} \right ) = 0\, ;\\
                                       \displaystyle \dt \left (\frac{v_y}{y^2}\right ) = - \frac{1}{y^3} \big (v_x^2+v_y^2 \big ) - \frac{v_x}{y^2}\, .
\end{array}\right .
\label{eq5.1.2}
\end{equation}

\begin{oss}
The system above shows that the momentum $p_x=\partial L/\partial v_x$ associated with $x$ is a prime integral, that is constant along the motion. The existence of two prime integrals (namely the energy $E$ and the momentum $p_x)$ can be used to show that, for every $k<\frac 12$, every Euler-Lagrange orbit with energy $k$ is periodic with period depending only on $k$.
\end{oss}

\noindent We recall that the \textit{Ma$\tilde{\text{n}}$\'e critical value} $c(L)$ is defined by 
$$c(L):= -\inf \left. \left \{\frac{1}{\tau}\int_0^\tau L(\gamma(t),\dot \gamma(t)\, dt \ \right |\ \gamma:\R/\tau \Z\stackrel{C^\infty}{\longrightarrow} \HH\right \},$$
or equivalently (cf. \cite{CI99} or  \cite{BP02}) by 
\begin{equation}
c(L):= \inf_{u\in C^\infty(\HH)} \sup_{q\in \HH} \, \frac{1}{2}\, \|du-\eta\|_q^2.
\label{mane}
\end{equation}

It is well-known that, for the Lagrangian $L$ in \eqref{Lag}, we have $c(L)=\frac 12$ (see for instance \cite[Section 6]{Ass15}). Moreover, the restriction of the Euler-Lagrange flow to the energy level set $E^{-1}(\frac 12)$ is 
the celebrated \textit{horocycle flow} (cf. \cite{Hed32} or \cite{BKS91} for further details): projected orbits are horospheres (i.e. spheres tangent to the boundary $\partial \HH)$ and the flow shifts inward pointing normal vectors of the horospheres to the left with constant speed 1.

\begin{oss}
The Euler-Lagrange flow defined by $L$ in \eqref{Lag} is an example of Euler-Lagrange flow with empty Aubry set. In particular, in sharp contrast with 
what happens for compact manifolds, there are no minimizing measures. We refer to \cite{CI99} for the precise definitions and the details. As we already observed, the dynamics for 
subcritical energies $k<\frac 12$ is periodic. For energies $k>\frac 12$, the dynamics is conjugated (up to time reparametrization) to the geodesic flow induced by a suitable Finsler metric on $\HH$ (see Appendix \ref{supercritical}). In particular, when projected to a compact quotient of $\HH$ (that is, a closed orientable surface with genus greater than or equal to 2), this flow provides an example where: 
\begin{itemize}
\item For energies $k<\frac 12$ the flow is periodic and every (periodic) orbit is contractible.
\item For energies $k>\frac 12$ the flow does not have contractible periodic orbits.
\item For $k=\frac 12$ there are no periodic orbits.
\end{itemize}
\end{oss}
\noindent The Hamiltonian $H:T^*\HH\rightarrow \R$ dual to $L$ is given by
\begin{equation}
H(q,p)= \frac{1}{2} \|p-\eta \|_q^2.
\label{hamiltoniana}
\end{equation}
Here,  for sake of simplicity, we denote the dual norm on $T^* \HH$ induced by the Riemannian metric $g$ also with $\|\cdot\|_q$. Using the standard coordinates 
$\big ((x,y),(p_x,p_y)\big )\in T^*\HH$ we obtain
\begin{eqnarray*}
H\big ((x,y),(p_x,p_y)\big ) = \frac{1}{2} \|p\|_q^2 - y\, p_x + \frac{1}{2}
\end{eqnarray*}
and hence the Hamilton-Jacobi equation $H(q,d_qu)=k$ can be written as
\begin{equation}
\frac{y^2}{2}\, |\nabla u|^2 - y \, \frac{\partial u}{\partial x}= k - \frac{1}{2}.
\label{hjgeneral}
\end{equation}

Studying the existence of (smooth) solutions of the Hamilton-Jacobi equation is particularly interesting when $k=c(L)$. In our case, for $k=\frac12$, \eqref{hjgeneral} becomes
\begin{equation}
y^2 \, |\nabla u|^2 -2y \, \frac{\partial u}{\partial x}= 0.
\label{HJ}
\end{equation}

The approach we will use to find the smooth solutions $u:\HH\rightarrow \R$ of \eqref{HJ} makes essential use of the theory of \textit{invariant Lagrangian graphs}, which we now 
recall. Let $M$ be a $n$-dimensional manifold, $T^*M$ its cotangent bundle equipped with the standard symplectic form $\omega=dp\wedge dq$.
A submanifold $W\subseteq T^*M$ is called \textit{Lagrangian} if for every $\xi\in W$ the tangent space $T_\xi W$ is a Lagrangian subspace (i.e. maximally isotropic) 
of $T_\xi T^*M$. Lagrangian submanifolds are of particular interest because of the following (cf. \cite{CI99})

\begin{teo}[Hamilton-Jacobi]
Let $W\subseteq T^*M$ be a Lagrangian submanifold. If $H|_W$ is constant, then $W$ is invariant under the Hamiltonian flow.
\label{teohj}
\end{teo}

Some distinguished and important $n$-dimensional submanifolds of $T^*M$ are graph submanifolds, which are of the form 
$$G_\alpha =\Big \{(q,\alpha_q)\ \Big |\ q\in M \Big \} \subseteq T^*M$$
where $\alpha$ is a 1-form on $M$; a \textit{Lagrangian graph} is a Lagrangian graph submanifold. By working in local coordinates one can prove the following (cf. \cite{CI99})
\begin{lemma}
$G_\alpha$ is Lagrangian if and only if $\alpha$ is closed.
\end{lemma}

Combining Theorem \ref{teohj} with the lemma above we get that the smooth solutions of the Hamilton-Jacobi equation \eqref{HJ} correspond to the exact (i.e. defined by an 
exact form) Lagrangian graphs contained in $\{H=\frac12\}$. 

We end this section noticing that smooth Hamilton-Jacobi solutions might not exist for a general Hamiltonian. This fact brought to the development of the 
so-called \textit{Weak-KAM theory} (see for instance \cite{Fat09} and \cite{Sor10} and references therein), a particularly fruitful approach to the study of the dynamical properties 
of the Hamiltonian system which is concerned with the study of the existence of less-regular (e.g. Lipschitz-continuous) solutions (and 
subsolutions) of the 
Hamilton-Jacobi equation.

%%%%%%%%%%%%%%%%%%%%%%%%%%%%%%%%%%%%%%%%%%%%%%%%%%%%%%%%%%%%%%%%%%%%%

\section{Proof of Theorem \ref{teogrosso}}
\label{proofteogrosso}
Observe preliminarly that the constant functions $u\equiv c$ are solutions of \eqref{HJ}; this implies that $\HH\times \{0\}$ is an exact invariant Lagrangian graph in $T^*\HH$. 
Another way to see this fact is the following: by (\ref{hamiltoniana}), we have 
$$H(q,0)= \frac 12 \, \|\eta\|_q^2 = \frac{1}{2};$$
therefore, the exact Lagrangian graph $G_0=\{(q,0)\, |\, q\in \HH\}$ is contained in $\{H=\frac 12\}$ and hence it is invariant by the Hamilton-Jacobi theorem \ref{teohj}.

After this simple observation we proceed with the study of the invariant Lagrangian graphs.
Notice that if an invariant graph is Lagrangian, then it is also exact; in fact, since $\HH$ is simply connected any closed 1-form in $\HH$ is exact. Moreover any invariant graph in $T^*\HH$ corresponds via the Legendre transform
\begin{equation}
\mathcal L:T\HH \longrightarrow T^*\HH\, , \ \ \ (q,v) \longmapsto \left (q, \frac{\partial L}{\partial v}(q,v)\right )= \left (q\ ,\left (\frac{v_x}{y^2} + \frac{1}{y}\, ,\, \frac{v_y}{y^2} \right ) \right )
\label{Legtransform}
\end{equation}
to an invariant (under the Euler-Lagrange flow) graph in $T\HH$; therefore it suffices to determine all the invariant graphs in $T\HH$ and
see which of them correspond to graphs in $T^*\HH$ that are also Lagrangian (i.e. defined by a closed 1-form). 

The first step in this direction is to show which is the shape of an invariant (under the Euler-Lagrange flow at level $\frac 12$) foliation of $\HH$; 
recall that the Euler-Lagrange flow at the energy level $\frac{1}{2}$ is the horocycle flow, hence projected solutions are horospheres. 
Let $\zeta : \HH \rightarrow T\HH$ be a vector field such that the graph 
$$\Gamma_\zeta = \Big \{ (q,\zeta(q))\ \Big |\ q\in \HH\Big \}$$ 
is invariant and contained in the energy level $\{E= \frac{1}{2}\}$. Clearly, $\zeta \neq 0$ everywhere and hence $\int \zeta$ 
defines an oriented foliation of $\HH$, which is composed by horospheres.

\begin{lemma}
Let $\mathcal F$ be a smooth foliation of $\HH$ invariant under the horocycle flow. Then, $\mathcal F$ is composed by horospheres tangent to the same point $a\in \partial \HH$; in particular, for $a=\infty$ the invariant foliation is composed by horizontal lines.
\label{lemma5.2.35}
\end{lemma}

\begin{proof}
Fix a point $a\in \partial \HH$; it is clear that the family of all horospheres tangent to $\partial \HH$ in $a$ forms an invariant foliation of $\HH$, since for any point $q\in \HH$ there is only one horosphere through $q$ and tangent to $\partial \HH$ in $a$. We denote by $\mathcal F_a$ such a foliation. 

Conversely, suppose that $\mathcal F$ is another invariant foliation of $\HH$ which is not of the form $\mathcal F_a$ for any $a\in \partial \HH$. 
Observe that if $\mathcal F$ contains a horosphere $h$ tangent to $\partial \HH$ in a point $b$ then $\mathcal F$ contains all horospheres 
tangent to $\partial \HH$ in $b$ that are contained in the disk $D_h$ bounded by $h$; in fact, if $q\in D_h$, the only horosphere through $q$ that does not 
intersect $h$ is the horosphere tangent to $\partial \HH$ in $b$. For every $b\in \partial \HH$ denote with $h_r^b$ the horosphere tangent 
to $\partial \HH$ in $b$ with (Euclidean) radius $r$. If $\mathcal F\neq \mathcal F_a$ for every $a\in \partial \HH$, then there exist 
$b\in \partial \HH$ and $R>0$ such that $\{h_r^b\}_{r< R}\subset \mathcal F$ and 
$h_r^b\not \in \mathcal F$, for every $r>R$.  Consider a point $q=(b,y)\in \HH$ with $y>2R$; then, the horosphere in $\mathcal F$ through $q$ must be the horizontal line $l_{y}$, for the only other horosphere through $q$ that does not intersect the family $\{h_r^b\}_{r< R}$ is $h_{y/2}^b$, which by assumption does not belong to $\mathcal F$. In particular, 
$$\big \{h_r^b\big \}_{r< R}\cup\big \{l_y\big \}_{y>2R}\subseteq \mathcal F.$$ 

If now $q'=(x,y')\in \HH \setminus \overline{D}_R$ is a point with $y'< 2R$, where $D_R$ is the disk bounded by $h_R^b$, then every horosphere through $q'$ 
has to intersect $\{h_r^b \}_{r< R}\cup \{l_y\}_{y>2R}$. In particular, $\mathcal F$  cannot be a foliation and this completes the proof.
\end{proof}

\vspace{3mm}

Observe that $\mathcal F_\infty$ is composed by horizontal lines, whilst $\mathcal F_0$ is composed by horospheres tangent to $\partial \HH$ in the origin; since a foliation determines the vector field $\zeta$ uniquely up to a multiplicative factor, we get that any invariant graph is of the form 
$$\alpha \cdot \Gamma_a =\Big \{ (q,\alpha\, v_q^a)\ \Big |\ q\in \HH\Big \}$$
where $\alpha$ is a positive number (being the foliation oriented) and $v_q^a$ is the unit tangent vector at $q$ to the horosphere through $q$ and tangent to $\partial \HH$ in $a$. 
Clearly, in our case $\alpha=1$, for we are considering invariant graphs contained in $\{E=\frac 12\}$; we write simply $\Gamma_a$ instead of $1\cdot \Gamma_a$. In the particular case $a=\infty$ we get 
$$\Gamma_\infty =\Big \{\big ((x,y),(-y,0)\big )\ \Big |\ (x,y)\in \HH\Big \} $$
%\begin{center}
%\begin{tikzpicture}
%\draw[->, thick] (0,-1) -- (12,-1) node[anchor=south west] {\Large{$\R$}};
%\draw (0.5,5) circle (0pt) node[anchor=west] {\huge{$\HH$}};
%\foreach \x in {0.12,0.25,0.37,0.5,0.75,1,1.5,2,3} {\draw[blue!30! black] (6,-1+\x) circle (\x);}
%\filldraw (6,-1) circle (1pt) node[anchor=north west] {\Large{$a$}};
%\end{tikzpicture}
%\vspace{5mm}
%\textbf{Figure 5.2} \textit{The smooth foliations of $\HH$ made by horospheres.}
%\end{center}
%\vspace{3mm}
with corresponding invariant graph $\mathcal L(\Gamma_\infty)=\HH \times\{0\}\subseteq T^*\HH$ that is clearly exact Lagrangian. This shows (in another way) that the constant functions are solutions of \eqref{HJ}.

Now consider $a\in \partial \HH \setminus \{\infty\}$; the invariant graph $\Gamma_a$ is 
$$\Gamma_a = \Big \{\big (\gamma_b^a(t), \dot{\gamma}_b^a (t)\big ) \ \Big |\ t\in \R, b\in \R\Big \}$$
where the $\gamma_b^a$'s are given by
$$\gamma_b^a(t)=\left (a + \frac{t}{b(1+t^2)}, \frac{1}{b(1+t^2)} \right )$$ 
and parametrize the horospheres tangent to $\partial \HH$ at $a$ by arc-length (see \cite{BKS91}). In particular the unit tangent vector $\dot{\gamma}_b^a(t)$
at $\gamma_b^a(t)$ is given by
\begin{equation}
\dot{\gamma}_b^a(t)\ =\ \frac{1}{b^2(1+t^2)^2}\, \big (b(1+t^2)-2bt^2, -2bt\big )\ =\ \frac{1}{b(1+t^2)^2}\, \big (1-t^2,-2t\big )\, .
\label{eq5.3.33}
\end{equation}

We want now to write $\Gamma_a$ with respect to the standard coordinates in $\HH$. Imposing the first coordinate of $\gamma_b^a(t)$ equal 
to $x$ and the second equal to $y$ yields
$$\ \ \ \ \ t=\frac{x-a}{y},\quad b = \frac{y}{(x-a)^2+y^2}\, .$$
Therefore
$$\dot{\gamma}_b^a(t)= b \, \big (y^2-(x-a)^2, -2(x-a)y\big )= \frac{y}{(x-a)^2+y^2}\, \big (y^2-(x-a)^2,-2(x-a)y\big )$$
and hence
\begin{equation}
\Gamma_a= \left. \left \{ \left ((x,y)\, ,\left (\frac{y}{(x-a)^2+y^2}\, \big (y^2-(x-a)^2,-2(x-a)y\big )\right )\right ) \ \right |\ (x,y)\in \HH\right \}\,.
\label{eq5.2.51}
\end{equation}
Using \eqref{Legtransform} we obtain that
$$\Sigma_a:= \mathcal L(\Gamma_a) =\left. \left \{ \left ((x,y)\, ,\left ( \frac{2}{(x-a)^2+y^2} \big (y,-(x-a) \big )\right ) \right )\ \right |\ (x,y)\in \HH\right \}\, .$$
is the invariant graph in $T^*\HH$ corresponding to $\Gamma_a$ via the Legendre transform. In particular,  $\Sigma_a =G_{\omega_a}$, where $\omega_a$ is the 1-form 
$$\omega_a(x,y)\ =\ \frac{2y}{(x-a)^2+y^2}\, dx - \frac{2(x-a)}{(x-a)^2+y^2}\, dy\, .$$
It is easy to check that the 1-form $\omega_a$ is closed (and hence exact), indeeed
$$\frac{\partial (\omega_a)_x}{\partial y} = \frac{2(x-a)^2-2y^2}{[(x-a)^2+y^2]^2}= \frac{\partial (\omega_a)_y}{\partial x}\, .$$
Therefore, there exists a smooth function $u_a:\HH\rightarrow \R$ such that $du_a=\omega_a$. The condition $\frac{\partial u_a}{\partial x} =(\omega_a)_x$ yields
$$u_a(x,y) = \int \frac{2y}{(x-a)^2+y^2}\, dx =\int \frac{2}{1+\zeta^2}\, d\zeta =2\, \arctan \zeta + g(y),$$ 
where we used the change of variable $\zeta = \frac{x-a}{y}$. Applying the condition $\frac{\partial u_a}{\partial y} = (\omega_a)_y$ we get   
$$\frac{\partial u_a}{\partial y}= \frac{-2(x-a)}{(x-a)^2+y^2} + g'(y) = \frac{-2(x-a)}{(x-a)^2+y^2} = (\omega_a)_y$$
and hence $g$ is a constant function. It follows that 
\begin{equation}
u_a(x,y)\ =\ 2\, \arctan \left (\frac{x-a}{y}\right )
\label{eq5.3.34}
\end{equation}
is a primitive of $\omega_a$. Hence, $\Sigma_a = G_{du_a}$ is an exact invariant Lagrangian graph and the function $u_a$ is a solution of the Hamilton-Jacobi equation. 
For sake of simplicity we also set $u_\infty =0$. Observe that there are no other smooth solutions of the Hamilton-Jacobi equation, besides those obtained by adding a constant to $u_a$; in fact, every solution corresponds to an invariant Lagrangian graph in $\{H=\frac 12\}$, which corresponds to a unique invariant graph in $\{E=\frac 12\}\subseteq T\HH$. This proves statement (1) of the theorem. As for the second statement 
of Theorem \ref{teogrosso}, with the argument above we have proved that any invariant graph contained in $\{H=\frac 12\}$ is exact Lagrangian, which is exactly what we wanted to prove.

We end this section showing a slightly different method to find the $\Sigma_a$'s, which is based on a "polar coordinates" parametrization of the $\Gamma_a$'s. The construction is the same for every $a\in \R$ (it suffices to "translate" the first coordinate in (\ref{eq5.2.50}) by the factor $a$), so that we can consider the invariant graph $\Gamma_0$ composed by the horospheres tangent to $\partial \HH$ at the origin and introduce the coordinate system $(r,\theta)$ for the hyperbolic plane 
\begin{equation}
\left \{\begin{array}{l} x=-r\, \sin \theta\, ;\\ \\ y = r \, (1-\cos \theta)\, ;\end{array}\right.
\label{eq5.2.50}
\end{equation}
where $r>0$ and $\theta \in (0,\pi)$. If we fix $r$, then the curve $\gamma_r(\theta)=r(-\sin \theta, 1-\cos \theta)$ parametrizes the horosphere of Euclidean radius $r$ tangent to $\partial \HH$ at the origin. Observe that $\gamma_r(\cdot)$ is in general not arc-length parametrized with respect neither to the Euclidean metric nor to the hyperbolic metric of $\HH$. However, the tangent vector at $\gamma_r(\theta)$ is given by 
$\gamma_r'(\theta) = r(-\cos \theta, \sin \theta)$, so that the (hyperbolic) unit tangent vector is 

$$v_r(\theta) \ =\  r(1-\cos \theta) \big (-\cos \theta, \sin \theta\big )\, .$$ 
\vspace{3mm}
By (\ref{eq5.2.50}) we get immediately that 
$$r = \frac{x^2+y^2}{2y}\, ,\ \ \ \ \sin \theta =\frac{-2xy}{x^2+y^2}\, ,\ \ \ \ \cos \theta= \frac{x^2-y^2}{x^2+y^2}$$ 
and hence 
\begin{eqnarray*}
\Gamma_0 &=& \left \{ \Big  ( r (-\sin \theta, 1-\cos \theta),\  r(1-\cos \theta) \big (-\cos \theta,\ \sin \theta) \Big ) \ \Big |\ \theta \in (0,2\pi),\, r>0\right \}\\
&=& \left. \left \{ \left ( (x,y)\ ,\frac{y}{x^2+y^2} (y^2-x^2, - 2xy) \right  ) \ \right | \ (x, y)\in \HH \right \}
\end{eqnarray*}
which is exactly (\ref{eq5.2.51}); now one completes the argument as above.

%%%%%%%%%%%%%%%%%%%%%%%%%%%%%%%%%%%%%%%%%%%%%%%%%%%%%%%%%%%%%%%%%%%%%

\section{On smooth Hamilton-Jacobi solutions for the geodesic flow}
\label{geodesics}
Here we want to compute, using the same method as in the last section, some particular smooth solutions of the Hamilton-Jacobi equation associated to the geodesic flow on $(\HH,g)$; recall that the geodesics of the hyperbolic plane are the semicircles orthogonal to the real line $\{y=0\}$ and the vertical lines. The geodesic flow can be seen  as the Euler-Lagrange flow associated with the kinetic energy 
$$L(q,v)\ =\ \frac{1}{2}\|v\|_q^2$$
where, as usual, $\|\cdot\|_q$ denotes the norm induced by the hyperbolic metric. Since the Euler-Lagrange flow restricted to $\{E=k\}$ is the geodesic flow up to a uniform change of speed, we can fix once for all the energy value $k=\frac{1}{2}$ (in which  geodesics are parametrized by arclength). The Hamilton-Jacobi equation $H(q,d_qu)=\frac{1}{2}$ has the form
$$H(q,d_qu)= \frac{1}{2}\, \|du\|_q^2= \frac{y^2}{2}\, \|du\|_E^2= \frac{1}{2}$$
where here $\|\cdot\|_q$ denotes the dual norm on $T^*\HH$ induced by the hyperbolic metric $g$ and $\|\cdot \|_E$ denotes the Euclidean norm. Thus, a smooth function $u:\HH\rightarrow \R$ is solution of the Hamilton-Jacobi equation if and only if 
\begin{equation}
|\nabla u|^2\ =\ \frac{1}{y^2}\, .
\label{HJeq2}
\end{equation}

We shall observe already at this point that the method we exploit might not find all the solutions of (\ref{HJeq2}), the reason being that we consider 
just particular invariant foliations of $\HH$. In the case of the horocycle flow (cf. Lemma \ref{lemma5.2.35}), the foliations made by horospheres tangent at a same point of $\partial \HH$ were the only possible; here this is no longer the case, as we will see by considering the invariant foliations composed by geodesics ending in the same point $a\in \partial \HH$ and the invariant foliations composed by geodesics with the same center $a\in \partial \HH$ (it is also possible to get other invariant foliations letting the center free to move).

Recall that in this case the Legendre transform is given by
\begin{equation}
(q,(v_x,v_y))\longmapsto \left (q, \frac{1}{y^2}(v_x,v_y)\right )
\label{Leggeo}
\end{equation}

Let us start with the invariant foliation of $\HH$ made by the vertical lines parametrized by arc-length, that is $\gamma_b(t)=(b,e^t)$ for any $b\in \R$, $t\in \R$, and consider the invariant graph 
$$\Gamma=\Big \{\big ((b,e^t),(0,e^t)\big )\ \Big |\ b\in\R\, ,\ t\in \R\Big \}=\Big \{\big ((x,y),(0,y)\big )\ \Big |\ (x,y)\in\HH\Big \}.$$
The corresponding invariant graph in $T^*\HH$ via the Legendre transform is 
$$\Sigma\ =\ \left. \left \{\left ((x,y), \left (0\, ,\frac{1}{y}\, \right )\right )\ \right|\ (x,y)\in \HH\right \}$$
and it is clear that $\Sigma$ is the graph of the form $du$, with $u=\log y$; hence $u$ is a solution of (\ref{HJeq2}). Since the geodesic flow is reversible, there is another invariant graph associated with the foliation made by vertical lines. Such a graph is obtained considering the "opposite" arclength parametrization of the vertical geodesics, 
that is $\gamma'_b(t)=(b,e^{-t})$, and yields  the solution $u=-\log y$ of the Hamilton-Jacobi equation. This fact could be also deduced from the quadratic dependence on (the derivatives of) $u$ in (\ref{HJeq2}) and implies that if we get a solution $u$ of the Hamilton-Jacobi equation then also $-u$ is solution. 

Now consider the foliation of $\HH$ made by geodesics ending at the same point $a\in \R$; the invariant graph associated to this foliation is (up to a time inversion) given by
$$\Gamma_a \ =\ \Big  \{\big (\gamma_a^b(t),\dot{\gamma}_a^b(t)\big )\ \Big |\ b,t\in \R\Big \}$$ 
where $\gamma_a^b(t)$ is the arclength parametrization of the geodesic with center in $a$ 
\begin{equation}
\gamma_a^b(t)= \left  ( a-\frac{b}{b^2+e^{2t}}, \frac{e^t}{b^2+e^{2t}} \right )\, .
\label{eqgeodetiche}
\end{equation}
In particular the tangent vector at $\gamma_a^b(t)$ is given by
$$\dot{\gamma}_a^b(t)= \frac{e^t}{(b^2+e^{2t})^2}\, \big (2e^tb, b^2-e^{2t}\big ).$$ 
In standard coordinates we have 
$$b=-e^t \, \frac{x-a}{y}, \quad e^t=\frac{y}{(x-a)^2+y^2}\, ;$$
therefore
\begin{equation*}
\Gamma_a \ =\ \left.  \left \{ \left ( (x,y )\, , \frac{y}{(x-a)^2+y^2} \left (\! -2(x-a)y, (x-a)^2-y^2\, \right  ) \right )\ \right |\ (x,y)\in \HH\right \}\, .
\end{equation*}
Using the Legendre transform \eqref{Leggeo} we obtain the corresponding invariant graph in $T^*\HH$
$$\Sigma_a\ =\ \left.\left \{\left ( (x,y), \frac{1}{(x-a)^2+y^2}\, \left (-2(x-a),\frac{(x-a)^2-y^2}{y}\right )\right ) \ \right |\ (x_1,x_2)\in \HH\right \}$$
which is the graph of the exact 1-form 
$$\omega_a(x,y)\ = \, -\,  \frac{2(x-a)}{(x-a)^2+y^2}\, dx \, +\,  \frac{(x-a)^2-y^2}{y[(x-a)^2+y^2]}\, dy\, ,$$
with primitive given by 
\begin{equation}
u_a(x,y)\ =\ \log y - \log \big ((x-a)^2+y^2\big )\,  .
\label{eq5.5.1}
\end{equation}
For sake of simplicity we also set $u_\infty(x,y):=\log y$. Observe finally that, if $u$ is a solution of \eqref{HJeq2}, then $u+c$ is also solution for every $c\in \R$.

\begin{prop}
For every $a\in \partial \HH$, for every $c\in \R$, the function $\pm \, u_a +c$ defined by (\ref{eq5.5.1}) is solution of the Hamilton-Jacobi equation (\ref{HJeq2}).
\end{prop}

Now we want to compute the Hamilton-Jacobi solutions associated with the invariant foliations of $\HH$ composed by geodesic half-circles with the same center $a\in \R$; the argument is the same for every $a\in \R$, so we may suppose $a=0$ and then shift the first coordinate by a factor $a$. Consider the polar coordinate system $(r,\theta)$ for the hyperbolic plane
$$\left \{\begin{array}{l} x = - r\, \cos \theta;\\ \\ y = r \, \sin \theta; \end{array}\right.$$ 
where $r>0$ and $\theta \in (0,\pi)$; for any $r$ fixed the curve $\gamma_r(\theta)=r (-\cos\theta,\sin \theta)$  parametrizes the geodesic with center in the origin and  Euclidean radius $r$. Although the parametrization is not by arc-length with respect neither to the Euclidean metric nor to the hyperbolic metric, we can associate to each point $\gamma_r(\theta)$ a (hyperbolic) unit tangent vector 
$$v_r(\theta)= r\, \sin \theta \, (\sin \theta, \cos \theta)\,.$$ 
Therefore we get the invariant graph 
\begin{eqnarray*}
\Gamma_0  &=& \Big \{ \big ( r \, (-\cos \theta, \sin \theta)\, ,\, r\, \sin \theta\, (\sin \theta,\cos \theta) \big )\ \Big |\ r>0,\, \theta\in (0,\pi) \Big \}\\ 
&=& \left.\left  \{\left ( (x,y)\, ,\, y\, \left (\frac{y}{\sqrt{x^2+y^2}}, - \frac{x}{\sqrt{x^2+y^2}} \right ) \right ) \ \right |\ (x,y)\in \HH \right \}\, .
\end{eqnarray*}
The corresponding graph in $T^*\HH$ via the Legendre transform is 
$$\Sigma_0 \ =\ \left. \left \{ \left ( (x,y)\, ,\, \left ( \frac{1}{\sqrt{x^2+y^2}}, \frac{-x}{y\, \sqrt{x^2+y^2}}\right ) \right ) \ \right |\ (x,y)\in \HH \right \}$$
which is the graph of the exact 1-form 
$$\omega_0(x,y)= \frac{1}{\sqrt{x^2+y^2}}\, dx \, - \, \frac{x}{y\, \sqrt{x^2+y^2}}\, dy\, ,$$
with primitive given by
$$u_0(x,y)= \text{arcsinh}\, \left ( \frac{x}{y} \right ).$$ 
Applying the translation $x\mapsto x+a$ on the first cordinate we get that, for any $a\in \R$, 
\begin{equation}
u_a(x,y)\ =\ \text{arcsinh}\, \left ( \frac{x-a}{y} \right )
\label{eq5.5.52}
\end{equation}
is a smooth solution of the Hamilton-Jacobi equation.  

\begin{prop}
For every $a\in \R$, for every $c\in \R$, the function $\pm \, u_a + c$ defined by (\ref{eq5.5.52}) is a smooth solution of the Hamilton-Jacobi equation (\ref{HJeq2}).
\end{prop}

\bibliographystyle{amsalpha}
\bibliography{bibliography}

\end{document}